\documentclass[12pt,a4paper]{amsart}
\usepackage{amsfonts,amsthm,amssymb,mathrsfs,amscd,amsmath,url}
\usepackage[utf8]{inputenc}
\usepackage[T1]{fontenc}
\usepackage{dsfont}

\usepackage{pdfcomment}

\usepackage[main=english]{babel}

\usepackage{colortbl}
\usepackage[utf8]{inputenc}
\usepackage[a4paper]{geometry}
\usepackage{natbib}
\usepackage{graphicx,caption}
\usepackage{comment}

\usepackage{varioref}

\usepackage{cleveref,autonum}
\usepackage{epigraph}

\usepackage{mathabx}
\usepackage{hyperref}

\makeatletter
\newcommand\incircbin
{%
  \mathpalette\@incircbin
}
\newcommand\@incircbin[2]
{%
  \mathbin%
  {%
    \ooalign{\hidewidth$#1#2$\hidewidth\crcr$#1\ovoid$}%
  }%
}

\makeatother

\usepackage{enumitem}
\usepackage{relsize}

\newcommand{\N}{\mathbb{N}}
\newcommand{\Z}{\mathbb{Z}}

\newcommand{\R}{\mathbb{R}}
\newcommand{\Q}{\mathbb{Q}}
\newcommand{\C}{\mathbb{C}}
\newcommand{\A}{\mathcal{A}}

\newtheorem*{theorem*}{Theorem}

\newtheorem*{remark}{Remark}
\newtheorem*{notation}{Notation}

\newtheorem{lemma}{Lemma}
\newtheorem*{lemma*}{Lemma}

\newtheorem{prop*}{Proposition}

\newtheorem*{corollary*}{Corollary}

\newcommand{\congruence}[3]{\ensuremath{{#1}\equiv {#2}\bmod{#3}}}
\newcommand{\res}[1]{\operatorname*{Res}_{#1}}

\newcommand{\1}[1]{\mathds{1}\left[#1\right]}

\newcommand{\mycomment}[1]{}
\renewcommand{\epsilon}{\varepsilon}
\renewcommand{\phi}{\varphi}

\renewcommand{\L}{\mathcal{L}}

\usepackage{xcolor}
\usepackage{comment}
\includecomment{comment}

\title{The largest prime factor of an irreducible cubic polynomial}
\author{Ivan Ermoshin}
\date{}

\begin{document}

\begin{abstract}
    Heath-Brown proved that for a positive proportion of integers $n$, $n^3+2$ has a prime factor larger than $n^{1+c}$ with $c=10^{-303}$.
    
    We generalize this result to arbitrary monic irreducible cubic polynomial of $\Z[x]$ with $c$ replaced by an exponent $c_p$ dependent on the polynomial.
\end{abstract}
\maketitle

\tableofcontents

\section{Introduction}
\subsection{Historical overview}
Let $f(X)\in\Z[X]$ be an irreducible polynomial with positive leading coefficient and no fixed prime divisor. There is a famous conjecture of Bunyakovsky that such polynomial has infinitely many prime values at integer arguments. It is far beyond our reach but some weakened versions has been proven for a general polynomials, and a lot more can be accomplished if one specifies the polynomial.

One approach is to consider the following quantity
$$P(x,f)=P^+\left(\prod_{n\le x}f(n)\right),$$
where for an integer $n$ we let $P^+(n)$ and $P^-(n)$ denote the largest and smallest prime factors of $n$, respectively.

The first such result was found by Chebyshev. It states that 
$$\frac{P(x,X^2+1)}{x}\to\infty,\text{ as }x\to\infty.$$
The proof was sketched in Chebyshev's posthumous manuscripts and published and proved in full by Markov in \cite{markov}.

The result was improved and generalized by Nagel \cite{nagell}, Erd\H{o}s \cite{erdos} and later Tenenbaum\cite{tenenbaum}. The following bound due to Tenenbaum is the best known result for polynomials of arbitrary degree strictly greater than 1
$$P(x,f)\gg x\exp((\log x)^A),\text{ where }A=2-\log4-\epsilon.$$

However for the specific case $f(X)=X^2+1$ Hooley\cite{hooleyquad} was able to show 
$$P(x,X^2+1)\gg x^{11/10}.$$
A few years later Hooley\cite{hooleycub} got the first bound for $P(x,f)$ with a cubic $f$, though conditional on certain estimates for short Kloosterman sums, that is
$$P(x,X^3+2)\gg x^{1+\delta}\text{ for }\delta=1/30.$$

Later Heath-Brown\cite{HB} gave an unconditional proof of this fact with $\delta=10^{-303}$. Moreover he proved that the set of integers $n$ such that $P^+(n^3+2)\gg n^{1+\delta}$ has positive density.
The constant $\delta$ was improved by Irving\cite{irving} and is the best known for the time being
$$P(x,X^3+2)\gg x^{1+\delta}\text{ with }\delta=10^{-52}.$$

Quartic polynomials have been considered by Dartyge\cite{dartyge}, La Bret\`eche\cite{breteche} and Dartyge and Maynard\cite{DM}.

To begin with Dartyge obtained the following result for the twelfth cyclotomic polynomial $\Phi_{12}(X)=X^4-X^2+1$
$$P(x,\Phi_{12}(X))\gg x^{1+\delta}\text{ with }\delta=10^{-26531}.$$

Then La Bret\`eche generalised the result for even monic quartic polynomials having Klein four group as its Galois group. The condition on Galois group is necessary for some polynomial decomposition properties that are discussed later.

Recently Maynard and Dartyge extended it further, namely for monic quartic polynomials with cyclic and dihedral Galois group.

We also note that, with the current state of the art, an explicit constant does not seem to be obtainable in such general results, since it depends on the size of a fundamental domain for the action of the unit group by multiplication on an extension of $\Q$ generated by a root of the polynomial, as well as on the residue at $s=1$ of the Dedekind zeta function of the splitting field of $f$ over $\Q$.

Heath-Brown in \cite{HB} posed a question if it is possible to generalize his result to arbitrary cubic polynomials. Even though the Galois group does not intervene in case of cubic polynomials and the decomposition difficulties mentioned earlier does not occur, there was no general result, which should be certainly possible after all the progress on degree 4. We establish it in this paper. 

\begin{theorem*}
    For any monic irreducible cubic polynomial $f\in\Z[X]$ exists a constant $c>0$ such that for at least a positive proportion of integers $n\in[x,2x]$, the number $f(n)$ has a prime factor exceeding $x^{1+c}$. In particular
    $$P(x,f)\gg x^{1+c}.$$
\end{theorem*}

\subsection{Notation} In the paper we will use letter $P$ for prime ideals in $\mathcal{O}_K$, where $K$ is a field to be defined, and letter $p$ for primes in $\Z$. All logarithms are to base $e$.
The signs $\ll$, $\gg$ are usual Vinogradov signs. The functions $(a,b)$ and $[a,b]$ stand for the greatest common divisor and the least common multiple of $a$ and $b$ respectively. Also we use the common notation for complex exponent:

$e(t)=\exp(2\pi it),$

$e_q(t)=\exp(2\pi it/q).$

\subsection{Description of the method}

Here we explain the method, prove a few preliminary lemmas and discuss the difficulties arising. Essentially it is but a generalization of Heath-Brown's argument described in section 2 in \cite{HB}.

Let $f(n)=x^3+c_2x^2+c_1x+c_0\in\Z[x]$ be an irreducible polynomial. Let $r$ be any of its roots, $K=\Q(r)$, $N(\cdot)=N_{K/\Q}(\cdot)$ the associated norm. Then we see that $N(n-r)=|f(n)|=f(n)$ for large $n$.  We are interested in counting integers $n$ such that the ideal $(n-r)$ has a prime factor of large norm. 

We shall work with the set $\A = \{ n-r : X < n < 2X \}\subset\mathcal{O}_K$ for a large $X$, 

\noindent say $X>3\max\limits_{0\le i\le 2} c_i$. For any ideal $I$  
$$\text{we define }\A_I = \#\{n-r\in\A:\;I\mid n-r\},$$
 
$$\text{and we define }\rho(I) = \#\{n \bmod N(I) : I\mid n-r \}.$$
The following lemma repeats lemma 1 from \cite{HB} and describes $\rho(\cdot).$ We shall prove it in subsection \ref{s2.2}.

\begin{lemma}\label{lemma 1}
Let $I$ be an ideal of $\Z[r]$ such that $(N(I),\text{Disc(f))=1}$.

If $\congruence{n-r}{0}{I}$ is solvable with a rational integer $n$, then $I$ is composed of first degree prime ideals only. If such $I$ is divisible by two distinct prime ideals of the same norm then $\rho(I)=0$; otherwise, the congruence is solvable, and we have $\rho(I) = 1$.

Moreover, if $I$ is an ideal for which $\rho(I)=1$, then for any $m\in\Z$, we have $I\mid m$ if and only if $N(I)\mid m$.
\end{lemma}

We should stress that the function $\rho(I)$ is not multiplicative. However we do have $\rho(IJ)=\rho(I)\rho(J)$, provided that $N(I)$ and $N(J)$ are coprime. By Lemma~\ref{lemma 1} is clear that
$$\#\mathcal{A}_{I}=\frac{\rho(I)}{N(I)}X+O(1)$$
for any $I$, and we shall define
$$R_{I}=\#\mathcal{A}_{I}-\frac{\rho(I)}{N(I)}X.$$

We may factor the ideal $(n-r)$ as
$$(n-r)=\prod_{P^e\mid n-r, N(P)\le3X}P^e \prod_{P^e\mid n-r, N(P)>3X}P^e.$$
Corresponding to this decomposition we write
$$\log(f(n)))=\log(N(n-r))=\log^{(1)}(f(n))+\log^{(2)}(f(n)), \text{ where}$$

$$\log^{(1)}(f(n)) = \sum_{\substack{I \mid n-r \\ \Lambda(I) \le \log 3X}} \Lambda(I),$$

$$\log^{(2)}(f(n)) = \sum_{\substack{I \mid n-r \\ \Lambda(I) > \log 3X}} \Lambda(I),$$
where $\Lambda(I)$ is the usual Von Mangoldt function over number fields
$$\Lambda(I)=\begin{cases}\log N(I),&I=P^k \\ 0,&\text{otherwise.}\end{cases}$$

Our principal task will be to construct a subset $\A_1\subset\A$, such that
$$\log^{(1)}(f(n))) > (1 + \delta)\log X \quad \text{for } n-r \in \A_1,$$
for some constant $\delta>0$. We shall also require that $\#\A_1 = X_1 =(\alpha+o(1))X$ for some real constant $\alpha>0$.

Now suppose that among the set $\A_2 = \A \setminus \A_1$ there are precisely $X_2$ elements with
$\log^{(1)}(f(n)) \ge (1 - \delta')\log X,$
where $\delta'$ is to be chosen later. It then follows that
\begin{equation}\begin{split}\sum_{n-r \in \A} \log^{(1)}(f(n)) \ge \sum_{n-r\in\A_1}\log^{(1)}(f(n))+\sum_{n-r\in\A_2}\log^{(1)}(f(n))\ge\\ X_1(1 + \delta)\log X + X_2(1 - \delta')\log X.\label{log1}\end{split}\end{equation}

On the other hand, since $f(n)\le9X^3$, we have
$$\sum_{n-r \in \A} \log^{(1)}(f(n))= \sum_{\substack{N(I) \le 9X^3 \\ \Lambda(I) \le \log 3X}} \Lambda(I) \cdot \#\A_{I} = \sum _{\substack{N(I) \le 9X^3 \\ \Lambda(I) \le \log 3X}}\Lambda(I) \left( \frac{\rho(I)}{N(I)}X + O(1) \right).$$

The error term is 
$$\ll \sum_{N(P)\le 3X}\log N(P)\sum_{e:\: N(P^e)\le 9X^3}1\ll \sum_{N(P)\le 3X}\log X\ll X$$
by the prime ideal theorem.

It is clear that the contribution of prime powers in the main term is negligible, therefore it is 
$$X\sum_{N(P)\le 3X}\frac{\rho(P)\log N(P)}{N(P)}=X\log X+O(X)$$
by the prime ideal theorem again. 

Thus
$$\sum_{n-r \in \A} \log^{(1)}(f(n))=X\log X+O(X),$$
and combining the above with (\ref{log1}) we get
\begin{equation}X+O(X/\log X)\ge X_1(1 + \delta)+ X_2(1 - \delta').\label{log2}\end{equation}

Now, if we set $\A_3 = \{ n-r \in \A : \log^{(1)}(f(n)) < (1 - \delta')\log X \}$, with $\#\A_3 = X_3 = X - X_1 - X_2$, replacing $X_2$ by $X-X_1-X_3$ in \eqref{log2} we get
$$X+O(X/\log X)\ge X_1(1 + \delta)+ (X-X_1-X_3)(1 - \delta').$$

This implies
$$X_3\ge X_3(1-\delta')\ge X_1(\delta+\delta')-X\delta'+O(X/\log X),$$
we shall choose $\delta'=\delta\frac{X_1}{X}+\epsilon$ with small $\epsilon>0$, so that $X_1\delta-\delta'X=-\epsilon X$ and
$$X_3\ge (\delta(X_1/X)^2-\epsilon)X+O(X/\log X).$$

Since for $n-r\in\A$
$$\sum_{I\mid n-r}\Lambda(I)=\log N(n-r)>(3+o(1))\log X,$$
we deduce that for any $n-r\in\A_3$
$$\log^{(2)}(f(n))=\log(f(n))-\log^{(1)}(f(n))>(2+\delta'+o(1))\log X.$$

The total number of all prime ideal factors $p$ counted with multiplicity by $\log^{(2)}(f(n))$ can be at most 2 since product of the norms of these ideals is $\le f(n)\le (3X)^3$. Then for any $n-r\in\A_3$ there must be a factor with 
$$\log N(p)\ge \frac{1}{2}\log^{(2)}f(n)\ge (1+\delta'/2+o(1))\log X,\text{ i.e. } N(p)\ge X^{1+\delta'/2+o(1)}.$$

If we pick $\epsilon$ small enough and $X$ large enough, $\epsilon$ cancels $o(1)$ in the power, while the density of $\A_3$ is still positive.

The following lemma summarizes the argument above and reflects Heath-Brown's lemma 2 from \cite{HB}.
\begin{lemma}
Let $\alpha$ and $\delta$ be positive constants. Suppose there exists a subset $\A_1 \subset \A$ as above. Then at least $(\alpha^2\delta + o(1))X$ integers $n \in (X, 2X)$ are such that $f(n)$ has a prime factor $p \gg X^{1 + \alpha\delta/2}$.
\end{lemma}

In order to construct the elements $n-r$ of $\A_1$ we arrange that $n-r$ has an ideal factor $J=KL$ with 

\begin{equation}X^{3\delta}<N(K)\leq X^{4\delta}.\label{norm_k}\end{equation}

\begin{equation}X^{1+\delta}<N(KL)\leq X^{1+2\delta}.\label{norm_kl}\end{equation}

For $K$ and $L$ as above we have
$$
N(L)=\frac{N(J)}{N(K)}\leq X^{1-\delta}.
$$

Since any factor $P$ of $J$ must divide either $K$ or $L$, it follows that $N(P)\leq X^{1-\delta}$. Thus

\begin{equation}\begin{split}\log^{(1)}(f(n))=\sum_{I\mid n-r,\Lambda(I)\leq\log 3X}\Lambda(I)
\geq\sum_{I\mid J,\Lambda(I)\leq\log 3X}\Lambda(I)
=\\\sum_{I\mid J}\Lambda(I)
=\log N(J)
\geq(1+\delta)\log X,\end{split}\end{equation}
as required.

We shall take $K$ to run over a set $\mathcal{K}$ of first degree prime ideals to be described, subject to (\ref{norm_k}). For each such $K$ we shall let $L$ run over a set $\mathcal{L}(K)$ to be described, subject to (\ref{norm_kl}). We avoid the fact that a given $n-r$ might occur many times as a multiple of a fixed $J$ by ensuring that $L$ is composed only of prime ideals $P$ with $N(P)>X^{\delta}$. More precisely, we sieve $L$ from below, to level $X^{\delta}$. Thus we take $\lambda_{d}$ to be the Rosser-Iwaniec weights for the lower bound sieve of dimension 1 and sieving limit $D=X^{3\delta}$, as described by Iwaniec \cite{iwaniec}, for example. These are supported on the square-free integers $d\leq X^{3\delta}$ and have the properties that

$$|\lambda_{d}|\leq 1,$$
and
\begin{equation}\sum_{d\mid n}\lambda_{d}\leq
\begin{cases}
1&\text{if }n=1,\\
0&\text{if }n\geq 2.
\end{cases}\label{rosser} \end{equation}

Moreover, for any non-negative multiplicative function $g(d)$ satisfying $g(p)<p$ for all primes, and

$$\prod_{w\leq p<z}\left(1-\frac{g(p)}{p}\right)^{-1}\leq\frac{\log z}{\log w}\left\{1+O\left(\frac{1}{\log w}\right)\right\}$$
for all $z>w\geq 2$, we have

$$\sum_{d: p\mid d\Rightarrow p<X^{\delta}}g(d)\lambda_{d}d^{-1}\geq\{C_{0}+o(1)\}\prod_{p<X^{\delta}}\left(1-\frac{g(p)}{p}\right), $$
where $C_{0}=\frac{2}{3}e^{\gamma}\log 2.$

We shall set $Q=\mathop{\prod\nolimits'}\limits_{p<X^{\delta}} p,$ the product being restricted to primes $p$ which split in $\Q(r)$. We then proceed to consider

$$S=\sum_{K\in\mathcal{K}}\sum_{L\in\mathcal{L}(K)}\left(\sum_{d\mid Q,N(L)}\lambda_{d}\right)\#\mathcal{A}_{KL}.$$

In view of (\ref{rosser}) we see that
$$S\leq\sum_{K\in\mathcal{K}}\sum_{\substack{L\in\mathcal{L}(K)\\ (N(L),Q)=1}}\#\mathcal{A}_{KL}$$$$
=\sum_{n-r\in\mathcal{A}}\#\{(K,L)\colon K\in\mathcal{K},L\in\mathcal{L}(K),(N(L),Q)=1,KL\mid n-r\}.$$

By construction, any $n-r$ which is counted with positive weight in the above sum will meet our requirements. However
$$f(n)\leq 9X^{3}<X^{(1+[3\delta^{-1}])\delta},$$
for large enough $X$. We therefore see that $n-r$ can have at most $[3\delta^{-1}]$ prime ideal factors $P$, counted according to multiplicity, for which $N(P)\geq X^{\delta}$. Similarly, there can be at most $[\delta^{-1}]$ prime ideal factors $P$ for which $N(P)>X^{3\delta}$. Moreover, if $\mathcal{A}_{KL}\neq\emptyset$ and $(N(L),Q)=1$ then $L$ must be composed of first degree prime ideals $P$ with $N(P)\geq X^{\delta}$. It follows that there are at most $[\delta^{-1}]$ possible choices for $K$, and at most $2^{[3/\delta]}$ possible choices for $L$, for any given choice of $n-r$. We have therefore produced at least $\delta 2^{-[3/\delta]}S$ suitable values of $n-r$.

On the other hand,
\begin{equation}
S=\sum_{K\in\mathcal{K}}\sum_{L\in\mathcal{L}(K)}\left(\sum_{d\mid Q,N(L)}\lambda_{d}\right)\left\{X\frac{\rho(KL)}{N(KL)}+R_{KL}\right\}=XS_{0}+S_{1},
\end{equation}\label{sec:test}
where
$$S_{0}=\sum_{K\in\mathcal{K}}\sum_{L\in\mathcal{L}(K)}\left(\sum_{d\mid Q,N(L)}\lambda_{d}\right)\frac{\rho(KL)}{N(KL)},$$
and
$$S_{1}=\sum_{K\in\mathcal{K}}\sum_{L\in\mathcal{L}(K)}\left(\sum_{d\mid Q,N(L)}\lambda_{d}\right)R_{KL}.$$

Adding up the analysis above we get the following result, which is lemma 3 from \cite{HB}.

\begin{lemma}
Suppose that $S_{1}=o(X)$. Then any constant $\alpha>0$ satisfying
$$
\alpha<\delta 2^{-[3/\delta]}S_{0}
$$
will be acceptable in Lemma 2.
\end{lemma}

So we have to provide for each $K\in\mathcal{K}$ a set $ \mathcal{L}(K)$ such that $S_0\gg1$ and $S_1=o(x)$.
\section{Auxiliary lemmae}
\subsection{Counting ideals}

    The following version of the prime ideal theorem is just enough for our purposes. For a reference one may use Elementary and Analytic Theory of Algebraic Numbers by Narkiewicz \cite{nark}.
    \begin{theorem*}Let $K$ be a number field and $N_K(\cdot)$ the corresponding norm, then we have
        $$\pi_K(x)=\#\{P\mid N_K(P)\le x\}=\frac{x}{\log x}(1+o(1)).$$
    \end{theorem*}
     The following lemma is quite standard, for a reference one may use Elementary and Analytic Theory of Algebraic Numbers by Narkiewicz \cite{nark} once again. 
    \begin{lemma}\label{resid} For a number field $K$ denote $\zeta_K(s)$ its zeta function, $N(\cdot)$ its norm and $h_K$ its class number. Then
    $$\sum_{\substack{\alpha\text{ principal}\\ N(\alpha)\le x}}\frac{1}{N(\alpha)}=\frac{1}{h_K}\res{s=1}\zeta_K(s)\log x+O(1).$$
    \end{lemma}

\subsection{Properties of $\rho$}\label{s2.2}

Here we prove Lemma \ref{lemma 1} describing
$$\rho(I) = \#\{n \bmod N(I) : I\mid n-r \}.$$
\begin{remark}
The condition $(N(I),\text{Disc}(f))=1$ is needed to ensure $I$ consists of prime ideals above non-ramified primes, which is necessary for Hensel lemma providing unique solutions to the congruences.
\end{remark}
Let $P$ be a prime ideal, containing $n-r$, i.e. such that $\congruence{n}{r}{P}$. Then any element of $\Z[r]$ is congruent to an integer ($\in\Z$). However $\Z[r]/P$ is a vector space with a generating set $1,r,r^2$, hence $1$ is a basis of $\Z[r]/P$, meaning it is one-dimensional vector space over $\mathbb{F}_p$ for certain $p\in P$. Therefore $N(P)=\#\Z[r]/P=p$.

Consider $P_1\neq P_2$ such that $P_i\mid I$ and $N(P_i)=p$.
Here we have that $p\Z[r]=P_1P_2P_3$ where $P_i=(p,n_i-r)$ with $n_i$ being roots of $f$ in $\Z/p\Z$. As $P_1\neq P_2$ we also have $n_1\not\equiv n_2\bmod p$, but $n_i\equiv r\bmod P_i$, and $n\equiv r\bmod P_i$ for $i=1,2$. This implies that $n_i\equiv n\bmod P_i$, while we can assume one of $n_i$, say $n_1$, is not congruent to $n$ modulo $p$. That is $n_1\not\equiv n\bmod p$ and $n_1\equiv n\bmod P_1$, together with the fact $P_1\cap\Z=p\Z$ we get a contradiction.

Now we prove the inverse direction. Let $I=\prod\limits_i P_i^{\alpha_i}$, such that $N(P_i)=p_i$ with $p_i\neq p_j$ if $i\neq j$. By chinese remainder theorem it is enough to consider the case $I=P^\alpha$, $N(P)=p$ and $P=(p,a-r)$ with $0\le a\le p-1$ simple root of $f$ modulo $p$.
Then we apply Hensel lemma to obtain an unique $0\le a_k\le p^k-1$ for each $k\in\N$ such that $\congruence{a_k}{a}{p}$ and $f(a_k)\equiv0\bmod p^k$. It implies that $P\mid a_k-r$ and $p^k\mid f(a_k)=N(a_k-r)$. $P$ is the only ideal of norm $p$ dividing $(a_k-r)$, then we conclude $P^k\mid a_k-r$, i.e. $\rho(P^k)= 1$.

It remains to show the equivalence announced at the end of the lemma: $I\mid m \iff N(I)\mid m$ for ideals with $\rho(I)=1$.

$\Leftarrow$ is obvious as $I\mid N(I)$. So we proceed on  $\Rightarrow$.

It is enough to consider $I=P^k$ with $N(P)=p$.
We prove that by induction. If $k=1$, $P\mid m\Rightarrow N(P)\mid N(m)= m^3$, as $N(P)$ is prime we see $N(P)\mid m.$

Assume now we have proven the equivalence up to $k-1$. We write $I=P^k=P\cdot P^{k-1}$. In the lemma we consider first degree prime ideals only, so the prime above $P$ decomposes as $p=PJ$ with $(P,J)=1$, therefore $P^{k-1}\mid m/p$, then by the induction hypothesis $N(P^{k-1})\mid m/p$, $N(P^k)\mid m.$

\subsection{The roots of $f$ modulo $m$}
This subsection closely follows subsection 5.2 from \cite{DM}.

For $\alpha\in\Z[r]$ we write $\alpha=a_0+a_1r+a_2r^2.$

Let $m_\alpha:\Q(r)\to\Q(r)$ be the multiplication map $m_\alpha(x)=\alpha x$. Let $M_\alpha$ be its matrix with respect to the basis $\{1,r,r^2\}.$

Since $r^3+c_2r^2+c_1r+c_0=0$ we have
\begin{equation}\label{M_alpha}
M_\alpha=\begin{pmatrix}
a_0 & -c_0a_2 & a_2c_0c_2 - a_1c_0\\
a_1 & a_0-c_1a_2 & a_2c_1c_2 - a_2c_0 - a_1c_1\\
a_2 & a_1-c_2a_2  & a_2c_2^2 - a_2c_1 - a_1c_2 + a_0
\end{pmatrix}.
\end{equation}

Let $B_{ij}=B_{ij}(\alpha)$ be the cofactors of $M_\alpha$, i.e. the determinant of the matrix formed by removing line $i$ and column $j$ from $M_\alpha$ multiplied by $(-1)^{i+j}$. It should be clear from the context if we mean a matrix or a polynomial in $a_0,a_1,a_2$ by $B_{ij}$.

The following is an analog of lemma 5.2 from \cite{DM}.
\begin{lemma}\label{5.2}
    If $\alpha$ is such that $(N(\alpha),B_{13}\text{Disc}(f))=1$ then there exists an integer $0\le k_\alpha<N(\alpha)$ such that 
    $$\congruence{n-r}{0}{\alpha}\iff \congruence{n}{k_\alpha}{N(\alpha)},$$
    also $k_{\alpha}$ satisfies the congruence
    $$k_{\alpha}\equiv B_{23}\overline{B_{13}}\bmod{N(\alpha)}.$$

    Furthermore, if $J$ is an ideal of $\Z[r]$ containing a principal ideal $(\alpha)$ with $\alpha$ as above then there exists $0\leq k_J<N(J)$ such that
    $$\congruence{n-r}{0}{J}\Leftrightarrow n\equiv k_J\bmod{N(J)}.$$
\end{lemma}

\begin{proof}
The starting point is the following trivial observation: $\alpha r^{j}\in(\alpha)$ for all $j=0,1,2,3,\ldots,n-1$. Let $(m_{i,j})_{1\leq i,j\leq n}$ be the entries of $M_{\alpha}$. We obtain the equations
$$m_{1,j}+m_{2,j}r+m_{3,j}r^2=0\bmod{(\alpha)},\;\forall\,1\leq j\leq n.$$

This system can be represented as
$$
\begin{pmatrix}
m_{2,1} & m_{3,1} \\
m_{2,2} & m_{3,2}  \\
m_{2,3} & m_{3,3} 
\end{pmatrix}
\begin{pmatrix}
r \\
r^{2} \\
\end{pmatrix}
=
\begin{pmatrix}
-m_{1,1} \\
-m_{1,2} \\
-m_{1,3}
\end{pmatrix}
\bmod{(\alpha)}
$$

If we remove the $i$-th line in this system and apply Cramer's rule, we find respectively for $i=1,2,3$ the following identities
$$
r\det
\begin{pmatrix}
m_{2,2} & m_{3,2} \\
m_{2,3} & m_{3,3}
\end{pmatrix}
=
\det
\begin{pmatrix}
-m_{1,2} & m_{3,2} \\
-m_{1,3} & m_{3,3} 
\end{pmatrix}
\bmod{(\alpha)},
$$
$$
r\det
\begin{pmatrix}
m_{2,1} & m_{3,1} \\
m_{2,3} & m_{3,3}
\end{pmatrix}
=
\det
\begin{pmatrix}
-m_{1,1} & m_{3,1}  \\
-m_{1,3} & m_{3,3} 
\end{pmatrix}
\bmod{(\alpha)},
$$
$$
r\det
\begin{pmatrix}
m_{2,1} & m_{3,1} \\
m_{2,2} & m_{3,2}
\end{pmatrix}
=
\det
\begin{pmatrix}
-m_{1,1} & m_{3,1}  \\
-m_{1,2} & m_{3,2} 
\end{pmatrix}
\bmod{(\alpha)}.
$$

The transpose of the matrix on the left is the submatrix of $M_{\alpha}$ obtained by removing the first line and the $i^{th}$ column. The matrix on the right is the submatrix of $M_{\alpha}$ obtained by removing the second line and the $i^{th}$ column and by multiplying all elements of the first column by $-1$.

We recall that the $B_{ij}$, $1\leq i,j\leq n$, are the cofactors of $M_{\alpha}$, so that
$$
M^{-1}_{\alpha}=\frac{1}{N(\alpha)}
\begin{pmatrix}
B_{11} & B_{21} & B_{31} \\
B_{12} & B_{22} & B_{32} \\
B_{13} & B_{23} & B_{33}
\end{pmatrix}.
$$

With this notation, the obtained identities becomes
\begin{equation}\label{5.5}
B_{1i}r\equiv B_{2i}\bmod{(\alpha)}.
\end{equation}

By Lemma \ref{lemma 1} (and the assumption $(N(\alpha),\,\mathrm{Disc}\,(P))=1$), if an integer is congruent to $0\bmod{(\alpha)}$ then it is divisible by $N(\alpha)$. Therefore considering $i=3$ now gives the claim of the first part of the present lemma.

For the second part when $J\mid (\alpha)$, thus it suffices to take $k_J\in[0,N(J)]$ such that $k_J\equiv k_{\alpha}\bmod{N(J)}$.
\end{proof}

We end this subsection by observing some connection between the cofactors $B_{1i}$ and $B_{2j}$ with $1\leq i,j\leq n$. Since $(m_{\alpha})^{-1}=m_{\alpha^{-1}}$, we have
$$\alpha^{-1}=\frac{1}{N(\alpha)}(B_{11}+B_{12}r+B_{13}r^2),$$
and the columns of $M_{\alpha}^{-1}$ satisfy the same relations \eqref{M_alpha} as the one in $M_{\alpha}$.

Hence we see that
\begin{equation}\label{B}
\begin{pmatrix}
B_{21} \\
B_{22} \\
B_{23}
\end{pmatrix}
=
\begin{pmatrix}
-c_0B_{13} \\
B_{11}-c_1B_{13} \\
B_{12}-c_2B_{13}
\end{pmatrix}.
\end{equation}

\subsection{Elimination of $a_0$}This is essentially the subsection 5.3 from \cite{DM}.

In preparation to apply Lemma \ref{HB} we approximate the fraction $k_J/N(\alpha)$ by a fraction whose denominator depends only on $a_1,a_2$. A natural way to proceed is to work with resultants of the polynomials defined previously.

\begin{lemma}
There is a homogeneous polynomial $q_0=q_0(a_1,a_2)$ in $a_1,a_2$ such that
$$B_{23}B_{11}-B_{13}B_{21}=q_0N(\alpha).$$
\end{lemma}

\begin{proof}
We note that the argument giving \eqref{5.5} holds for any $\alpha\neq 0$. Applying this with $i=1,3$ we find$$
rB_{11}B_{23}\equiv rB_{21}B_{13}\bmod{N(\alpha)}.$$

Since this holds for all $a_0,a_1,a_2$ with $(N(\alpha),\mathrm{Disc}\,(P))=1$, we deduce that there exists a form $q_0=q_0(a_0,a_1,a_2)$ such that for $a_0,a_1,a_2$ satisfying $(N(\alpha),\mathrm{Disc}\,(P))=1$ we have
$$B_{23}B_{11}-B_{13}B_{21}=q_0N(\alpha).$$

Since both sides are polynomials and the set of triples $(a_0,a_1,a_2)\in\Z^3$ satisfying the coprimality condition is Zariski dense in $\mathbb{A}_\Q^3$, this identity must actually hold for all $\alpha$ including $(N(\alpha),\mathrm{Disc}\,(P))\neq 1$. Therefore we just need to show that $q_0$ actually doesn't depend on $a_0$. The polynomial $N(\alpha)$ has degree 3 in $a_0$ while the polynomials $B_{ij}$ are of degree 1 if $i\neq j$ in $a_0$ and of degree 2 if $i=j$, and so by equating the coefficients of $a_0^3$ we see that $q_0$ must not depend on $a_0$.
\end{proof}

\begin{remark}
One can explicitly compute $q_0$ in terms of the coefficients $c_{i}$ of $f$; it is given by
$$q_0(a_1,a_2)=a_2c_2 - a_1.$$
\end{remark}

Following the notation of \cite{breteche} and \cite{dartyge}, we write $\mathrm{Resultant}(P_1,P_2;x)$ for the resultant of the polynomials $P_1$, $P_2$ with respect to the variable $x$. We will be interested by the two following resultants
$$R:=R(a_1,a_2)=\mathrm{Resultant}(B_{13},N(\alpha);a_0),$$
$$R_0:=R_0(a_1,a_2)=\mathrm{Resultant}(B_{13},B_{22};a_0).$$

The following lemma is an analog of lemma 5.4 from \cite{DM}. While we could not find a conceptual proof, one may check this identity holds by direct calculation using SAGE.
\begin{lemma}
With the previous notation we have
$$-q_0^{2}R=R_0^{2}.$$
\end{lemma}

\begin{remark}Here we try to repeat the proofs of analogous lemmas from \cite{DM} and \cite{breteche}.

Since $B_{13}$ is of degree $1$ in $a_0$, we have
$$q_0 R=\mathrm{Resultant}(B_{13},q_0N(\alpha);a_0)=\mathrm{Resultant}(B_{13},B_{23}B_{11}-B_{13}B_{21};a_0).$$
But by \eqref{B} we have $B_{11}=B_{22}+c_1B_{13}$, $B_{23}=B_{12}-c_2B_{13}$, so
$$q_0R=\mathrm{Resultant}(B_{13},B_{12}B_{22}+B_{13}(-B_{21}+c_1B_{23}-c_2B_{22});a_0)=$$$$\mathrm{Resultant}(B_{13},B_{12}B_{22})=R_0 \cdot\mathrm{Resultant}(B_{13},B_{12}).$$
So we need to prove
$$-q_0\mathrm{Resultant}(B_{13},B_{12};a_0)=\mathrm{Resultant}(B_{13},B_{22};a_0)=R_0,$$
which we can not do by algebra, but it can be done by direct computation.
\end{remark}

We see that the polynomial $q_0$ divides $R_0$, and so we can write
$$R_0=qq_0$$
for some homogeneous polynomial $q=q(a_1,a_2)$. Moreover, since $R_0$ is the resultant of $B_{13}$ and $B_{22}$, there are two polynomials $U$ and $V\in\Z[a_0,a_1,a_2]$ such that
$$UB_{22}+VB_{13}=qq_0.$$

We are now ready to state the main result of this section.

\begin{lemma}\label{kalpha}
Suppose $a_0,a_1,a_2$ are such that $(B_{13}(a_0,a_1,a_2),q(a_1,a_2))=1$. Then $(N(\alpha),B_{13}(a_0,a_1,a_2))=1$ and we have
$$e\Big(\frac{k_{\alpha}}{N(\alpha)}\Big)=e\Big(\frac{U(a_0,a_1,a_2)\overline{B}_{13}(a_0,a_1,a_2)}{q(a_1,a_2)}+E(a_0,a_1,a_2)\Big),$$
where $E$ is given by
$$E(a_0,a_1,a_2)=-\frac{U}{qB_{13}}+\frac{B_{23}}{N(\alpha)B_{13}}.$$
\end{lemma}

\begin{proof}
To simplify notation, for the proof let $q,q_0,U,B_{14},B_{14},B_{23},B_{24},N(\alpha)$ denote the values of the polynomials evaluated at $a_0,a_1,a_2$.

Since $R=-q^2$, if $q$ is coprime with $B_{13}$, we have $(N(\alpha),B_{13})=1$. 

\noindent By Lemma \ref{5.2},
$$
e\Big(\frac{k_{\alpha}}{N(\alpha)}\Big)=e\Big(\frac{B_{23}\overline{B}_{13}}{N(\alpha)}\Big).
$$

We use the Bezout relation
$$
\frac{\overline{u}}{v}+\frac{\overline{v}}{u}\equiv\frac{1}{uv}\bmod{1}\quad\text{for }(u,v)=1,
$$
and the fact that $(N(\alpha),B_{13})=1$. This yields the formula
$$
e\Big(\frac{k_{\alpha}}{N(\alpha)}\Big)=e\Big(-\frac{B_{23}\overline{N(\alpha)}}{B_{13}}+\frac{B_{23}}{B_{13}N(\alpha)}\Big).
$$

Also we obtain
$$UN(\alpha)q_0=U[B_{23}(B_{22}+c_1B_{13})-B_{13}B_{21}]=UB_{13}(c_1B_{23}-B_{21})+UB_{22}B_{23}=$$$$UB_{13}(c_1B_{23}-B_{21})+B_{23}(R_0-VB_{13})=UB_{13}(c_1B_{23}-B_{21})-VB_{13}B_{23}+B_{23}qq_0.$$

Then, as $q_0$ and $B_{13}$ are coprime as polynomials, we have the implication 
$$\congruence{q_0(UN(\alpha)-qB_{23})}{0}{B_{13}}\Rightarrow \congruence{UN(\alpha)}{qB_{23}}{B_{13}}.$$

Using this and Bezout once more we deduce
$$e\Big(\frac{k_{\alpha}}{N(\alpha)}\Big)= e\Big(-\frac{U\overline q}{B_{13}}+\frac{B_{23}}{B_{13}N(\alpha)}\Big)=e\Big(\frac{U\overline B_{13}}{q}-\frac{U}{qB_{13}}+\frac{B_{23}}{B_{13}N(\alpha)}\Big).$$
\end{proof}

There are explicit expression for important polynomials from this subsection.
$$q=a_2^3c_1c_2 - a_1a_2^2c_2^2 - a_2^3c_0 - a_1a_2^2c_1 + 2a_1^2a_2c_2 - a_1^3.$$
$$q_0=a_2c_2 - a_1.$$
$$U=a_2^2.$$
$$V= a_2\cdot a_0+(a_2^2c_2^2 - 2a_1a_2c_2 + a_1^2).$$
$$B_{13}=- a_2\cdot a_0 + (a_2^2c_1 - a_1a_2c_2 + a_1^2). $$
$$B_{22}=a_0^2 + a_0(-a_2c_1 - a_1c_2) +(-a_2^2c_0c_2 + a_0a_2c_2^2 + a_1a_2c_0).$$

\begin{remark}
    For $f=X^3+2$ we obtain $q=-a_1^3-2a_2^3$, while Heath-Brown in \cite{HB} has an analogous polynomial equal to $a_0^3-2a_1^3$ due to different choice of indices and eliminating $a_2$ in the denominator.
\end{remark}

\subsection{Fundamental domain for the action of the group of units of $\Q(r)$.}\label{fd}
Let $r_1$ and $2r_2$ be numbers of real and complex embeddings of $\Q(r)$ into $\C$ respectively.

Due to Dirichlet unit theorem there are two cases of how the unit group $E$ looks like:
\begin{enumerate}
    \item $r_1=r_2=1$ and $E=E_0\times \langle w\rangle$ with $w>1$, where $E_0$ is the group of roots of unity in $\Q(r)$. Here we simply put $\mathcal{D}$ to be the set of the $\alpha$ with
    $$1\le \frac{|\alpha|}{|N(\alpha)|^{1/3}}\le |w|.$$
    
    \item $r_1=3,\;r_2=0$ and $E=\{\pm1\}\times \langle w_1,w_2\rangle$ with $1<w_1<w_2$. In this group, the totally positive units form a subgroup of index at most $2^{\#\text{S}_3}=2^6$. Denote as $w_1^+, w_2^+$ its generators.

    Now we define $\mathcal{D}$ as $\mathcal{D}=\mathcal{D}_1\cup\mathcal{D}_2$ with $\mathcal{D}_1,\mathcal{D}_2$ being the open cones generated over $\R_+$ by $\{1,w_1^+,w_1^+w_2^+\}$ and $\{1,w_2^+,w_1^+w_2^+\}$ respectively.
\end{enumerate}

In the first case it is easy to see that $\mathcal{D}$ is $\#E_0$ copies of the fundamental domain.

In the second case due to \cite{colmez} and \cite{diaz} $\mathcal{D}$ is also union of a finite number $d$ of copies of the fundamental domain.

\begin{lemma}\label{D}
    Let $\alpha\in\Z[r]$. Then there are finitely many $\alpha'\in\mathcal{D}$ such that $(\alpha)=(\alpha')$ as ideals.
\end{lemma}
\begin{proof}
    While the second case $\Q(r)\subset\R$ is described in \cite{colmez}, the first one is simple. We can assume $\alpha\in\mathcal{D}$. Also we have $|N(\alpha)|=|N(\alpha')|$, therefore there is a unit $\theta  $ such that $\alpha'=\theta\alpha$. The definition of $\mathcal{D}$ implies $1/w<|\theta|<w$ and consequently $N(\theta)=1$, $|\theta|=1$ and $\theta\in E_0$. 
\end{proof}

\begin{lemma}\label{norma}
    Let $\alpha=a_0+a_1r+a_2r\in\mathcal{D}\cap\Z[r]$, then $N(\alpha)>0$ and
    $$\max\{|a_0|,|a_1|,|a_2|\}\ll N(\alpha)^{1/3}.$$
\end{lemma}
\begin{proof}
    In the case (1) there is $\tau\in\text{Gal}(\Q(r)/\Q)$ such that $N(\alpha)=\alpha|\tau(\alpha)^2|$, while the condition on $\mathcal{D}$ implies $|\tau(\alpha)|\le|\alpha|$ and thus
    $$\max\{|\alpha|,|\tau(\alpha)|\}\le |w|N(\alpha)^{1/3}.$$
    
    The conclusion follows from equivalence of the usual norm on $\R^3$ and the norm $\max\{|\alpha|,|\tau(\alpha)|\}$.

    In the case (2) we note that $\text{Gal}(\Q(r)/\Q)=\Z/3\Z$. Now we consider only $\alpha\in\mathcal{D}_1$ as proof for $\mathcal{D}_2$ is the same. We have 
    $$\alpha=x_1\cdot1+x_2\cdot w_1^++x_3\cdot w_1^+w_2^+$$ with integers $x_i\ge0$ and
    $$N(\alpha)=\prod_{\tau\in\text{Gal}(\Q(r)/\Q)}(x_1\tau(1)+x_2\tau(w_1^+)+x_3\tau(w_1^+w_2^+))\ge\sum_{i=1}^3 x_i^3$$
    after developing the product and observing that $\tau(w_i^+)>0$ and $N(w_i^+)=1$.

    Then $\max x_i\le N(\alpha)^{1/3}$ and due to equivalence of the norms we obtain $\max\{a_0,a_1,a_2\}\ll N(\alpha)^{1/3}.$
\end{proof}

\subsection{A q-Van der Corput estimate for short exponential sums}

Here we state the key lemma for estimating $S_1$, this is the theorem 2 from \cite{HB}.

\begin{lemma}\label{HB}
Let $k \geq 1$, $D \geq 1$ and $\epsilon > 0$. Let $f, g \in \Z[X]$ polynomials of degree $\leq D$ and $q = q_0 \cdots q_k$ a squarefree integer with no prime factors $\le 2^k D$. Suppose for every $p \mid q$, there is no polynomial $w \in \Z[X]$ of degree $\le k + 1$ for which $f(X) \equiv w(X)g(X) \bmod{p}$ or $v(X) \equiv 0 \bmod{p}$ holds. Then, for $A, B, h \geq 1$, we have
$$
\sum_{\substack{A<n\leq A+B \\ (g(n),q)=1}} e\left(\frac{hf(n)\overline{g}(n)}{q}\right)
$$$$\ll_{k,D,\epsilon} q^{\epsilon}B\left(\left(\frac{\Delta}{q_0}\right)^{1/2^{k+1}} + \left(\frac{q_0}{\Delta B^2}\right)^{1/2^{k+1}} + \sum_{j=1}^{k}\left(\frac{q_{k+1-j}}{B}\right)^{1/2^j}\right)
$$
with $\Delta = (q_0, h)$.
\end{lemma}

Below we check the condition of the lemma in our case: $f(m):=U(\tilde a_0+mN(KA),a_1,a_2)$, $g(m):=B_{13}(\tilde a_0+mN(KA),a_1,a_2)$, where $KA\mid (\alpha)$ and $(N(KA),q)=1$.
\begin{lemma}
    Let $\alpha=a_0+a_1r+a_2r^2$ and $q$ squarefree with $P^-(q)>256$, $(a_1a_2,q)=1$ and let $p\mid q$ while $p\not\mid\text{Disc}(f)$. Then there is no polynomial $l$ of any degree such that $$U\equiv B_{13}l\bmod p.$$
\end{lemma}
\begin{proof}
    Recall 
    $$U= a_2^2,$$
$$B_{13}=- a_2\cdot a_0 + (a_2^2c_1 - a_1a_2c_2 + a_1^2).$$

    Then $U$ is of degree exactly 0, $B_{13}$ is of degree exactly 1, then the congruence in the lemma can not be true as identity over $a_0$.
\end{proof}

\section{Preliminaries of the proof}

The initial stages in our treatment of the sums $S_0$ and $S_1$ are the same, and will be described in this section. First we specify the sets $\mathcal{K}$ and $\mathcal{L}(K)$. Recall that
$$M=X^{(1+\delta)/3}, \;\;\;\;N= X^{(1+2\delta)/3}.$$

First, for some parameter $\delta>0$ we define $\mathcal{K}$ as the set of first degree prime ideals $K$ with 
\begin{enumerate}
    \item $X^{3\delta}<N(K)<X^{4\delta}$,
    \item $(N(K),\text{Disc}(f))=1$.
\end{enumerate}
Then $L\in\mathcal{L}(K)$ are ideals with $KL=(\alpha)=(a_0+a_1r+a_2r^2)$ such that\label{def l}
\begin{enumerate}
    
    \item  $ q\gg M^3$, $q$\text{ is squarefree}, $P^-(q)>\max(256,c_0)$,  $(a_1,a_2)=1$, $(a_1a_2,q)=1$,
    \item $\exists\text{ primes }q_1,q_2\mid q$ with $N^{5/7}<q_1<N^{5/7+\delta}$, $N^{6/7}<q_2<N^{6/7+\delta}$,
    \item $X^{1+\delta}<N(KL)\leq X^{1+2\delta}$,
    \item $\rho(\alpha)=1$ and $\alpha\in\mathcal{D},$
    \item $B_{13}\gg M^2$, $(q,B_{13})=1$, $(N(\alpha), B_{13}\text{Disc}(f))=1,$
    \item $2\not\mid a_0$,
    \item $P^-(N(L))>X^\theta$.
\end{enumerate}
Also we denote by $\mathcal{C}$ the set of pairs $(a_1,a_2)$, satisfying the first two conditions from the definition of $\mathcal{L}(K)$ (note they does not involve $a_0$):
\begin{enumerate} 
    \item  $ q\gg M^3$ $q$\text{ is squarefree}, $P^-(q)>\max(256)$,  $(a_1,a_2)=1$, $(a_1a_2,q)=1$.
    \item $\exists$ primes $q_1q_2\mid q$ with $N^{5/7}<q_1<N^{5/7+\delta}$, $N^{6/7}<q_2<N^{6/7+\delta}$,
\end{enumerate}
and as $\mathcal{R}\subset\R^3$ the set of real triples $(a_0,a_1,a_2)$ satisfying non arithmetic conditions on $\alpha$:
\begin{enumerate} \label{defr}
    \item $\alpha\in\mathcal{D},$
    \item $q\gg M^3$,
    \item $B_{13}\gg M^2,$
    \item $X^{1+\delta}<N(\alpha)\leq X^{1+2\delta}$.
\end{enumerate}
\begin{remark}We should mention that $\mathcal{L}(K)$ can contain multiple $L$ with the same $(\alpha)$, but only finitely many thanks to Lemma \ref{D}. And the terms in the sums $S_0$ and $S_1$ actually depend on $\alpha$, not on $L$, so indexing the sums $S_0$ and $S_1$ with $K\in\mathcal{K},\alpha\in\mathcal{R}$ is equivalent to the original indexing with $K\in\mathcal{K},L\in\mathcal{L}(K)$ for our purpose as it changes the sums by finitely many times.\end{remark}

We recall the sums $S_0$ and $S_1$ defined in the \hyperref[sec:test]{section 1.3}

$$ S_0 = \sum_{K \in \mathcal{K}} \sum_{L \in \mathcal{L}(K)} \left( \sum_{d\mid Q, N(L)} \lambda_d \right) \frac{\rho(KL)}{N(KL)} $$
and
$$ S_1 = \sum_{K \in \mathcal{K}} \sum_{L \in \mathcal{L}(K)} \left( \sum_{d\mid Q, N(L)} \lambda_d \right) R_{KL}. $$

According to Lemma \ref{5.2}, we will have $\rho(KL) = 1$ for every $L \in \mathcal{L}(K)$. We may therefore introduce a factor $\rho(KL)$ into the sum $S_1$. Since $\rho(KL) = 1$, we see that $L$ is composed of first degree prime ideals. 

Let
$ R = \prod\limits_{2 < N(P) < X^\delta} P, $
the product being restricted to first degree primes.

Recall 
$Q=\prod\limits_{p<X^{\delta}}p,$ the product being restricted to primes which split in $\Q(r)$.

We proceed to show that in the sums $S_0,S_1$ we may take $d$ to run over all square-free values of $N(A)$, where $A \mid  R, L$. To prove this we let $(L, d) = A$. Then $A \mid  Q$, so that $A$ must be composed of first degree prime ideals $P$ with $N(P) < X^\delta$. We also have $(2, KL) = 1$, whence $(2, A) = 1$. Since $\rho(KL) = 1$, we have $\rho(A) = 1$. Thus $A$ cannot have two distinct prime factors of the same norm. Thus $A$ must divide $R$, and $N(A)$ must be square-free. Since $A \mid  d$, we have $N(A) \mid  d^3$, whence $N(A) \mid  d$. On the other hand, if $p$ is a prime factor of $d$, then $p \mid  N(L)$, whence $L$ must have a prime ideal factor $P$ of norm $p$. Then $P \mid  L, d$ so that $P \mid  A$ and $p \mid  N(A)$. It follows that $d \mid  N(A)$, and hence that $N(A) = d$. Thus each value of $d$ arises as $N(A)$. Conversely we note that if $A \mid  R$ and $N(A)$ is square-free, then $A \mid  L$ implies $N(A) \mid  Q, N(L)$. Hence each possible $A$ produces an admissible value $d = N(A)$. This establishes the result claimed above.

It now follows that
$$ S_0 = \sum_{K \in \mathcal{K}} \sum_{L \in \mathcal{L}(K)} \left( \sum_{A\mid R,L} \lambda_{N(A)} \right) \frac{\rho(KL)}{N(KL)} $$
and
$$ S_1 = \sum_{K \in \mathcal{K}} \sum_{L \in \mathcal{L}(K)} \left( \sum_{A\mid R,L} \lambda_{N(A)} \right) \rho(KL) R_{KL}. $$

We see from the condition \eqref{norm_k} for $\mathcal{K}$ and the fact that $\lambda_d$ is supported on $d \leq X^{3\delta}$ that $K$ and $A$ can be taken to be coprime. For every $L \in \mathcal{L}(K)$ we have $\rho(KL) = 1$ by Lemma \ref{5.2}. We may therefore write
$$ S_0 = \sum_{K \in \mathcal{K}} \sum_{A\mid R} \lambda_{N(A)} \rho(KA) \sum_{\substack{L \in \mathcal{L}(K), A\mid L}} N(KL)^{-1} $$
and
$$ S_1 = \sum_{K \in \mathcal{K}} \sum_{A\mid R} \lambda_{N(A)} \rho(KA) \sum_{\substack{L \in \mathcal{L}(K), A\mid L}} R_{KL}. $$

\section{Bounding $S_1$}

According to what we have proved so far for $I$ satisfying $\rho(I)=1$ we have
$$\A_I = \#\{n-r\in\A:\;I\mid n-r\}=\#\{n\mid X<n\le 2X, \; \congruence{n}{k_I}{N(I)}\}=$$$$\#\left\{m\mid \frac{X-k_I}{N(I)}<m\le\frac{2X-k_I}{N(I)}\right\}=\frac{X}{N(I)}+\psi\left(\frac{X-k_I}{N(I)}\right)-\psi\left(\frac{2X-k_I}{N(I)}\right),$$

where $\psi(t)=t-[t]-1/2$. Then for $R_I$ we have

 $$R_I=\#\mathcal{A}_{I}-\frac{\rho(I)}{N(I)}X=\psi\left(\frac{X-k_I}{N(I)}\right)-\psi\left(\frac{2X-k_I}{N(I)}\right).$$
 Here we use the partial Fourier series for the fractional part.
 \begin{lemma}
    Let $H>0$, then
    $$\psi(t)=-\sum_{n\le H}\frac{\sin(2\pi nt)}{\pi n}+O\left(\min\{1,(H\lVert t\rVert)^{-1}\}\right).$$
 \end{lemma}

Now that
$$ S_1 = \sum_{K \in \mathcal{K}} \sum_{A\mid R} \lambda_{N(A)} \rho(KA) \sum_{\substack{L \in \mathcal{L}(K), A\mid L}} R_{KL}, $$
we take a few more conditions out of $L\in\mathcal{L}(K)$ and consider the innermost sum in
$$S_1=\sum_{K \in \mathcal{K}} \sum_{\substack{A\mid R, N(A)\le X^{3\delta}\\ (N(A),\text{Disc}(f))=1}} \lambda_{N(A)} \rho(KA) \sum_{\substack{L \in \mathcal{L}(K), A\mid L}} R_{KL}.$$

We expand $R_{KL}$ as difference of $\psi$-s and apply the Fourier cut-off of $\psi$ with $H=X^{\eta}$ with $\eta\le1/10$. 

Then the sum over $n\le H$ contributes
$$\ll \sum_{n\le H}n^{-1}\left|\sum_{L\in\mathcal{L}(K),A\mid L}\sin\left(2\pi n\frac{X-k_\alpha}{N(\alpha)}\right)-\sin\left(2\pi n\frac{2X-k_\alpha}{N(\alpha)}\right)\right|$$
$$\ll\sum_{n\le H}n^{-1}(|\Sigma_1(n)|+|\Sigma_2(n)|),$$
where $$\Sigma_j(n)=\sum_{L\in\mathcal{L}(K),A\mid L}e_{N(\alpha)}\left(n(jX-k_\alpha)\right).$$

While for the error term we expand $\min$ as Fourier series
$$\min\{1,(H\lVert t\rVert)^{-1}\}=\sum_{n=-\infty}^\infty c_ne(nt),$$
and using trivial bounds and integration by parts we find that $$c_n\ll\min\left\{\frac{\log H}{H},\frac{H}{n^2}\right\}.$$
Then the error contributes
\begin{equation}\begin{split}\ll\sum_{\substack{L\in\mathcal{L}(K),A\mid L\\ j=1,2}}\min\left\{1,\frac{1}{H\lVert(jX-k_\alpha)/N(\alpha)\rVert}\right\}=\sum_{\substack{L\in\mathcal{L}(K),A\mid L\\ j=1,2}}\sum_{n=-\infty}^\infty c_n e_{N(\alpha)}\left(n(jX-k_\alpha)\right)\\\ll\sum_{\substack{n\\j=1,2}} \min\left\{\frac{\log H}{H},\frac{H}{n^2}\right\} |\Sigma_j(n)|.\end{split}\end{equation}

Now we take $n=0$ out of the sum and cut the rest by $H^2$. For the tail $n>H^2$ we estimate the exponent inside of $\Sigma_j(n)$ trivially by 1, thus by
$\sum_{n>H^2}{H}{n^{-2}}\ll H^{-1}$
the sum over $n>H^2$ contributes the same magnitude as the term $n=0$.
\begin{equation}\begin{split}\label{rkl}\sum_{\substack{L \in \mathcal{L}(K), A\mid L}} R_{KL}\ll (\log H)H^{-1}\sum_{\substack{L \in \mathcal{L}(K), A\mid L}} 1+(\log H)\sum_{\substack{n=1\\ j=1,2}}^\infty \min\{n^{-1},Hn^{-2}\}|\Sigma_j(n)| \\\ll
(\log H)H^{-1}\sum_{\substack{L \in \mathcal{L}(K), A\mid L}} 1+(\log H)\sum_{\substack{n=1\\ j=1,2}}^{H^2} \min\{n^{-1},Hn^{-2}\}|\Sigma_j(n)|.\end{split}\end{equation}

Now we apply Lemma \ref{5.2} with $J=KA$, it implies that there exists an integer $k_{KA}$ such that
    $$\congruence{n-r}{0}{KA}\Leftrightarrow n\equiv k_{KA}\bmod{N(KA)}.$$ 
Then $\alpha\equiv \congruence{a_0+a_1r+a_2r^2}{0}{KA}$ is equivalent to
$$\congruence{a_0}{-k_{KA}a_1-a_2k_{KA}^2}{N(KA)}.$$
So $a_0$ runs over an arithmetic progression modulo the norm of the ideal. Thus
$$\sum_{\substack{L \in \mathcal{L}(K), KA\mid (\alpha)}} 1\ll N^2\left(\frac{N}{N(KA)}+1\right)$$
as $|a_i|\ll N(\alpha)^{1/3}\ll N,$
so the first term in (\ref{rkl}) contributes to $S_1$
\begin{equation}\begin{split}(\log H)H^{-1}N^2\sum_{K \in \mathcal{K}} \sum_{\substack{A\mid R, N(A)\le X^{3\delta}}} |\lambda_{N(A)}| \rho(KA) \left(\frac{N}{N(KA)}+1\right)\\\ll
(\log H)H^{-1}N^2\sum_{K \in \mathcal{K}} \sum_{\substack{A\mid R, N(A)\le X^{3\delta}}} \left(\frac{N}{N(KA)}+1\right)\\ \ll(\log H)H^{-1}N^2 (N(\log X)^2+X^{4\delta+3\delta})\ll (\log H)H^{-1}N^3(\log N)^2.\label{final s1 1}\end{split}\end{equation}

Now we turn our attention to the exponential sums. Write 
$$E_j(n)=\sum_{L\in\mathcal{L}(K),A\mid L}e\left(\frac{njX}{N(\alpha)}-\frac{nU\overline B_{13}}{q}\right),$$
then by Lemma \ref{kalpha}

$$\Sigma_j(n)-E_j(n)=\sum_{L\in\mathcal{L}(K),A\mid L}e\left(\frac{njX-nk_\alpha}{N(\alpha)}\right)\left(1-e\left(\frac{nk_\alpha}{N(\alpha)}-\frac{nU\overline B_{13}}{q}\right)\right)$$
$$\ll\sum_{L\in\mathcal{L}(K),A\mid L}\left|-\frac{nU}{qB_{13}}+\frac{nB_{23}}{N(\alpha)B_{13}}\right|,$$
by $B_{23},U\ll N^2$, $B_{13}\gg M^2$, $q,N(\alpha)\gg M^3$ we get
$$\ll \frac{nN^2}{M^5}\sum_{L\in\mathcal{L}(K),A\mid L}1.$$
Therefore to change $\Sigma_j(n)$ for $E_j(n)$ costs
$$ \sum_{K \in \mathcal{K}} \sum_{\substack{A\mid R, N(A)\le X^{3\delta}}} |\lambda_{N(A)}| \rho(KA) (\log H)\sum_{\substack{n=1\\ j=1,2}}^{H^2} n\cdot\min\{n^{-1},Hn^{-2}\}\frac{N^2}{M^5}\sum_{L\in\mathcal{L}(K),A\mid L}1$$
$$\ll\frac{N^2}{M^5}\cdot H(\log H)^2\cdot N^3(\log N)^2=H(\log H)^2\frac{N^5}{M^5}(\log N)^2\label{final s1 2}$$
where we used the calculation from (\ref{final s1 1}). Clearly the obtained term is negligible for the whole $S_1$.

Now we are going to apply Lemma \ref{HB} to show that for suitable parameters $S_1=o(X)$.

First for any $a_1,a_2$ we put $\mathcal{R}(a_1,a_2)=\{a_0\mid (a_0,a_1,a_2)\in\mathcal{R}\}\subset\R$ and write
$$E_j(n)=\sum_{K \in \mathcal{K}} \sum_{\substack{A\mid R, N(A)\le X^{3\delta}}} \lambda_{N(A)} \rho(KA)
\sum_{(a_1,a_2)\in\mathcal{C}}\sum_{\substack{a_0\in\mathcal{R}(a_1,a_2)\\ a_0\equiv \tilde a_0 \bmod N(KA)\\ (B_{13},q)=1}}e\left(\frac{njX}{N(\alpha)}-\frac{nU\overline B_{13}}{q}\right)$$
with $\tilde a_0=\tilde a_0(a_1,a_2; KA)\equiv -k_{KA}(a_1+a_2 k_{KA})\bmod N(KA)$, where $k_{KA}$ is an integer defined by lemma \ref{lemma 1} and lemma \ref{5.2}.

Now we make change of variables $$a_0=\tilde a_0+mN(KA),$$
consider $t=(N(KA),q)$, $t'=q/t$.
\begin{notation}
    Bars below denote the inverse with respect to the denominator. 
\end{notation}
The sum over $a_0$ is empty whenever $(B_{13},t)>1$, so we can assume $(B_{13},t)=1$. Taking $(u,v)=(t,t')$ in
$$
\frac{\overline{u}}{v}+\frac{\overline{v}}{u}\equiv\frac{1}{uv}\bmod{1}\quad\text{for }(u,v)=1,
$$
we deduce
$$e\left(-\frac{nU\overline B_{13}}{q}\right)=e\left(-\frac{nU\overline{t'B_{13}}}{t}-\frac{nU\overline{tB_{13}}}{t'}\right).$$ 

We also have
$$U(\tilde a_0+mN(KA),a_1,a_2)\overline {B_{13}(\tilde a_0+mN(KA),a_1,a_2)}\equiv U(\tilde{\mathbf{a}})\overline{ B_{13}(\tilde{\mathbf{a}})} $$
with $\tilde{\mathbf{a}}:=(\tilde a_0,a_1,a_2)$, so the first term does not depend on $m$. Put
$$f(m)=U(\tilde a_0+mN(KA),a_1,a_2),\;\; g(m)=B_{13}(\tilde a_0+mN(KA),a_1,a_2).$$
Thus we get
\begin{equation}\begin{split}E_j(n)=\sum_{K \in \mathcal{K}} \sum_{\substack{A\mid R, N(A)\le X^{3\delta}}} \lambda_{N(A)} \rho(KA)\cdot\\\sum_{(a_1,a_2)\in\mathcal{C}}e\left(-\frac{nU(\tilde{\mathbf{a}})\overline{t'B_{13}(\tilde{\mathbf{a}})}}{t}\right)\sum_{\substack{m\in\mathcal{R}'(a_1,a_2)\\ (g(m),t')=1}}e\left(\frac{njX}{N(\alpha)}-\frac{nf(m)\overline{t g(m)}}{t'}\right),\end{split}\end{equation}

where $\mathcal{R}'(a_1,a_2)=\{m\mid \tilde a_0+mN(KA)\in\mathcal{R}(a_1,a_2)\}$. 

The only obstacle to apply lemma \ref{HB} is $e\left(njX/N(\alpha)\right)$, to get rid of it we say use summation by parts. Note that $\mathcal{R'}(a_1,a_2)$ is a finite union of intervals contained inside $(-O(N/N(KA));O(N/N(KA)))$.

Obviously we have
$$\partial N(\alpha)^{-1}/\partial a_0\ll N(\alpha)^{-4/3}.$$
We get
$$\sum_{\substack{m\in\mathcal{R}'(a_1,a_2)\\ (g(m),t')=1}}e\left(\frac{njX}{N(\alpha)}-\frac{nf(m)\overline{t g(m)}}{t'}\right)\ll \left(1+\frac{nXN}{M^4}\right)\max_{B, B'\ll N/N(KA)}\left|\sum_{\substack{B'\le m\le B'+B\\ (g(m),t')=1}} e\left(\frac{n\overline tf(m)\overline{ g(m)}}{t'}\right)\right|.$$
Note that $XNM^{-4}\le 1$.

Now we apply lemma \ref{HB} with $k=2,D=1$ and the factorization $$t'=\frac{q/(q_1q_2)}{(q/(q_1q_2),t)}\cdot\frac{q_1}{(q_1,t)}\cdot\frac{q_2}{(q_2,t)},$$
it yeilds
$$\max_{B\ll N/N(KA)}\left|\sum_{\substack{m\le B\\ (g(m),t')=1}} e\left(\frac{n\overline tf(m)\overline{ g(m)}}{t'}\right)\right|\ll $$
$$\ll\frac{N^{1+\epsilon}}{N(KA)}\left(\left(\frac{\Delta(t',q_1q_2)}{t'}\right)^{1/8}+\left(\frac{t'N(KA)}{\Delta N^2(t',q_1q_2)}\right)^{1/8}+\left(\frac{(t',q_2)N(KA)}{N}\right)^{1/2}+\left(\frac{(t',q_1)N(KA)}{N}\right)^{1/4}\right)$$
where $\Delta=\left(\frac{t'}{(t',q_1q_2)};n\overline t\right)$,
$$\ll \frac{N^{1+\epsilon}}{N(KA)}\left(\left(\frac{(q,n)q_1q_2N(KA)}{q(q_1q_2,N(KA))}\right)^{1/8}+\left(\frac{qN(KA)}{ N^2q_1q_2}\right)^{1/8}+\left(\frac{q_2N(KA)}{N}\right)^{1/2}+\left(\frac{q_1N(KA)}{N}\right)^{1/4}\right).$$
Applying the trivial bounds and the following ones:
$$N^{3\delta}<N(KA)<N^{7\delta},$$
$$N^{5/7}<q_1<N^{5/7+\delta},$$
$$N^{6/7}<q_2<N^{6/7+\delta},$$
$$M^3\ll q\ll N^3,$$
we get
$$\max_{B\ll N/N(KA)}\left|\sum_{\substack{m\le B\\ (g(m),t')=1}} e\left(\frac{n\overline tf(m)\overline{ g(m)}}{t'}\right)\right|\ll(q,n)N^{13/14-17/24\delta+\epsilon}\ll(q,n)N^{13/14+\epsilon}.$$
We lost some $\delta$, though it does not matter much. Thus
$$(\log H)\sum_{\substack{n=1\\ j=1,2}}^{H^2} \min\{n^{-1},Hn^{-2}\}|E_j(n)|\ll(\log H)N^{3-1/14+\epsilon}\sum_{n=1}^{H^2}\min\{1,H/n\}(q,n)\ll$$
$$\ll(\log H)HN^{3-1/14+2\epsilon}.$$

Putting everything together
$$S_1\ll (\log H)H^{-1}N^3(\log N)^2+(\log H)HN^{3-1/14+7\delta+3\epsilon}.$$

Taking $H=N^{1/28-7/2\delta}$ we get
$$S_1\ll N^{4\epsilon}( H^{-1}N^{3}+HN^{3-1/14+7\delta})\ll N^{3-1/28+7/2\delta}=X^{(3-1/28+7/2\delta)(1+2\delta)/3}=$$$$N^{83/84+22/7\delta+7/3\delta^2}.$$
So $S_1=o(X)$ for $\delta\le10^{-3}.$

\section{Treating $S_0$}

Here we mostly follow Heath-Brown's paper \cite{HB}, thas is section 6,7 and 8, though we slightly change non essential steps due to the different setting. At the beginning we take some steps from Dartyge and Maynard \cite{DM}. 

Recall
$$ S_0 = \sum_{K \in \mathcal{K}} \sum_{A\mid R} \lambda_{N(A)} \rho(KA) \sum_{\substack{L \in \mathcal{L}(K), A\mid L}} N(KL)^{-1}, $$
we want to isolate the variable $a_0$. We note that the condition $L\in\L(K)$ implies that $(q, B_{13})=1$
and that $(a_0, a_1, a_2) \in \mathcal{R}$, where $\mathcal{R}$ is defined in \hyperref[defr]{section 3}, but otherwise there are no further dependencies between $a_0$ and $a_1, a_2$. We use M\"obius inversion to detect the condition $(q, B_{13}) = 1$, this gives rise to a squarefree $r \mid (q, B_{13})$ which we decompose as $r = r_1r_2$ with $r_1 \mid N(KA)$ and $(r_2, N(KA)) = 1$. Therefore

$$ S_0 = \sum_{K \in \mathcal{K}} \sum_{A} \lambda_{N(A)}\sum_{(a_1,a_2)\in\mathcal{C}\cap\mathcal{R}'}\sum_{r_1\mid (N(KA),q)}\mu(r_1)\sum_{\substack{(r_2,N(KA))=1\\ r_2\mid q}}\mu(r_2) \times$$$$\times\sum_{\tilde a\in S(r_1,r_2)}\sum_{\substack{a_0: (\alpha)\in\mathcal{R}\\ \congruence{a_0}{\tilde a_0}{r_2N(KA)}}}\frac{1}{N(\alpha)},$$
where $\mathcal{R}'=\{(a_1,a_2)\mid \exists a_0 \text{ s.t. }(a_0,a_1,a_2)\in\mathcal{R}\}$ is the projection of $\mathcal{R}$ and
$$S(r_1,r_2)=\{0\le a_0< r_2N(KA), \text{ such that } r_1r_2\mid B_{13}\text{ and } KA\mid (\alpha)\}.$$

Let $I(a_1,a_2)=\int\limits_{(\alpha)\in\mathcal{R}}\frac{da_0}{N(\alpha)}$, then by partial summation the inner sum is
$$\sum_{\substack{a_0: (\alpha)\in\mathcal{R}\\ \congruence{a_0}{\tilde a_0}{r_2N(KA)}}}\frac{1}{N(\alpha)}=\frac{I(a_1,a_2)}{r_2N(KA)}+O(X^{-1-\delta}).$$

The $O(X^{-1 - \delta})$ error term contributes to $S_0$ a total
$$\ll \frac{1}{X^{1 + \delta - o(1)}} \sum_{N(K) \ll X^{4\delta}} \sum_{N_P(A) \le X^{3 \delta}} \sum_{(a_1, a_2) \in \mathcal{C}} 1 \ll X^{-1/3 + 8 \delta}.$$

Now we calculate $\#S(r_1,r_2)$. By Chinese reminder theorem $$\#S(r_1, r_2) = \prod\limits_{p \mid r_1r_2N(KA)} \#S(r_1, r_2, p),$$
where
$$S(r_1,r_2,p)=\begin{cases}
    \{0\le a_0< p,\text{ such that }p\mid(N(\alpha),B_{13})\}&\text{ if }p\mid r_1,\\
    \{0\le a_0< p,\text{ such that }p\mid B_{13}\}&\text{ if }p\mid r_2,\\
    \{0\le a_0< p,\text{ such that }p\mid N(\alpha)\}&\text{ if }p\mid N(KA)/r_1.
\end{cases}$$

We now prove that in all the cases we have $\#S(r_1,r_2,p)=1$.

The second case is easy, as $B_{13}$ is linear in $a_0$ with leading coefficient equal to $a_2$, so with leading coefficient coprime to $q$.

The first case follows from $p\mid q$ and $q\mid R=\text{Resultant}(B_{13},N(\alpha);a_0)$, hence $B_{13}$ and $N(\alpha)$ has a common root modulo $p$ and it is the only one as $B_{13}$ is linear.

The third case follows from lemma \ref{lemma 1} and lemma \ref{5.2}.

Therefore $\#S(r_1,r_2)=1$. Hence

$$S_0= \sum_{K \in \mathcal{K}} \sum_{A} \lambda_{N(A)}\sum_{(a_1,a_2)\in\mathcal{C}\cap\mathcal{R}'}\sum_{r_1\mid (N(KA),q)}\mu(r_1)\sum_{\substack{(r_2,N(KA))=1\\ r_2\mid q}}\mu(r_2) \frac{I(a_1,a_2)}{r_2N(KA)}+o(1).$$
The sum over $r_1$ implies $(q,N(KA))=1$.  For convenience we define $g(d)=\#\{I: N(I)=d\}.$ Now we compute the sums over $A$ and $K$.

$$\sum_{\substack{K\in\mathcal{K}\\ (q,N(K))=1}}\frac{1}{N(K)}=\sum_{K\in\mathcal{K}}\frac{1}{N(K)}+O(X^{-3\delta+o(1)})=$$$$\sum_{X^{3\delta< p<X^{4\delta}}}\frac{g(p)}{p}+O(X^{-3\delta+o(1)})=\log\frac{4}{3}+o(1)$$
by the prime ideal theorem.

And as for $A$ we have

$$\sum_{(N(A),q)=1}\frac{\lambda_{N(A)}}{N(A)}=\sum_{\substack{d\le X^{3\delta}\\ (d,q)=1}}\frac{\lambda_d g(d)}{d}\ge (C_0+o(1))\prod_{p\mid q}\left(1-\frac{g(p)}{p}\right)^{-1}\prod_{p\le X^\delta}\left(1-\frac{g(p)}{p}\right).$$
Putting these expressions together yields
$$S_0\ge (C+o(1))\prod_{p\le X^\delta}\left(1-\frac{g(p)}{p}\right)\sum_{(a_1,a_2)\in\mathcal{C}\cap\mathcal{R}'}I(a_1,a_2)h(q),$$
where we have defined
$$h(q)=\mu^2(q)\prod_{p\mid q}\left(1-\frac{g(p)}{p}\right)^{-1}\left(1-\frac{1}{p}\right)\1{P^-(q)>\max(256,c_0)}.$$

Now we split $\mathcal{R}$ into small boxes 
$$\mathcal{B}=[A,A+M]\times[B,B+M]\times[C,C+M],$$
we call such a box good if $\mathcal{B}\subseteq\mathcal{R}$. Then 
\begin{equation}\label{sogb}\sum_{(a_1,a_2)\in\mathcal{C}\cap\mathcal{R}'}I(a_1,a_2)h(q)\ge \sum_{\text{good }\mathcal{B}}\sum_{(a_1,a_2)\in\mathcal{C}\cap B'}I(a_1,a_2)h(q),\end{equation}
where $'$ stands for projection again.

Now, with $a_i$ in a small box, the integral $I(a_1,a_2)$ can be approximated with no dependence on $a_i$ themselves
$$I(a_1,a_2)=\frac{M}{N(A,B,C)}(1+o(1)).$$

For the part with calculating the sum with $h(q)$ we follow Heath-Brown's lattice points approach from \cite{HB}.

\begin{equation}\sum_{(a_1,a_2)\in\mathcal{C}\cap B'}h(q)=\sum_{\substack{N^{5/7}<q_1<N^{5/7+\delta}\\ N^{6/7}<q_2<N^{6/7+\delta}}}\sum_{\substack{B\le a_1\le B+M\\ C\le a_2\le C+M\\ q_1q_2\mid q}}h(q)\ge\frac{1}{3}\sum_{q_1,q_2}C(q_1q_2),\label{cqq}\end{equation}
 $$\text{where }C(d)=\sum_{\substack{B\le a_1\le B+M\\ C\le a_2\le C+M\\ d\mid q}}h(q).$$

Now we state three lemmas, generalizing lemmas 10, 11 and 7 respectively from \cite{HB}, without the proof, as it is almost the same.

\begin{lemma}
    For real $A,B\ll N$, $M$ and $N$ as before and an ideal $R$ 
    
    \noindent let $S(R)=\#\{A\le a\le A+M, B\le b\le B+M, R\mid a-br\}.$ Then
    $$\sum_{N(R)\le X, \rho(R)=1}\left|S(R)-\frac{M^2}{N(R)}\right|\ll (M+X)X^\epsilon.$$
\end{lemma}

\begin{lemma}
    For real $A,B\ll N$, $M$ and $N$ as before and an ideal $R$  
    
    \noindent let $T(R)=\#\{A\le a\le A+M, B\le b\le B+M, (a,b)=1, R\mid a-br\}.$ 
    
    \noindent Set $\gamma(R)=\prod\limits_{p\mid N(R)}\left(1+\frac{1}{p}\right)^{-1}$. Then
    $$\sum_{N(R)\le X,}\left|T(R)-\frac{6}{\pi^2}\frac{M^2}{N(R)}\gamma(R)\rho(R)\right|\ll (NX^{1/2}+N^{3/2})N^\epsilon$$
    for any $X\le N^2$.
\end{lemma}

Define multiplicative functions $l(\cdot)$ and $\nu(\cdot)$ as follows.
$$l(p^e)=\begin{cases}\frac{g(p)-1}{p-g(p)},&e=1, p>256,\\ -1,&e=1, p\le256,\\ -h(p)&e=2,\\ 0&e\ge3.\end{cases}$$

$$\nu(p^e)=\frac{g(p)}{1+p^{-1}}.$$
One can easily check that $h=1*l$.

Then the next lemma follows from the previous two.
\begin{lemma}
    Let $A,B\ll N$, $M,N$ as before, write $$C(m)=\sum_{\substack{A\le a<A+M\\ B\le b<B+M\\ m\mid q}}h(q).$$
    Suppose that $(0,0)\not\in[A;A+M]\times[B;B+M]$, then if
    $$C_1=\frac{6}{\pi^2}\sum_{d=1}^\infty l(d)\frac{\nu(d)}{d}$$
    we have
    $$\sum_{q_1,q_2}\left|C(q_1q_2)-C_1M\frac{\nu(q_1)\nu(q_2)}{q_1q_2}{}\right|\ll M^2N^{-\delta}.$$
\end{lemma}

Substituting this bound to (\ref{cqq}) we get
$$\sum_{(a_1,a_2)\in\mathcal{C}\cap B'}h(q)\ge \frac{1}{3}C_1M^2\sum_{q_1,q_2}\frac{\nu(q_1)\nu(q_2)}{q_1q_2}+O(M^2N^{-\delta}),$$
by the prime ideal theorem we have
$$\sum_{p\le x}\frac{\nu(p)}{p}=\log\log x+C+o(1),$$
thus
$$\sum_{q_1,q_2}\frac{\nu(q_1)\nu(q_2)}{q_1q_2}=\log\left(1+\frac{7}{5}\delta\right)\log\left(1+\frac{6}{5}\delta\right)+o(1),$$
denote this constant as $L(\delta)$, then
$$\sum_{a_1,a_2}h(q)\ge\frac{1}{3}C_1M^2(1+o(1))L(\delta).$$

Now we put it into the sum over good boxes (\ref{sogb}) and get
$$\sum_{(a_1,a_2)\in\mathcal{C}\cap\mathcal{R}'}I(a_1,a_2)h(q)\ge \frac{1}{3}C_1L(\delta)(1+o(1))\sum_{\text{good } \mathcal{B}}\frac{M^3}{N(A,B,C)}.$$
Now for a good box $\mathcal{B}$
$$\frac{M^3}{N(A,B,C)}=(1+o(1))\sum\limits_{(a_0,a_1,a_2)\in \mathcal{B}}\frac{1}{N(a_0,a_1,a_2)},$$
and for the whole sum
$$\sum_{\text{good } \mathcal{B}}\frac{M^3}{N(A,B,C)}=(1+o(1))\sum\limits_{(a_0,a_1,a_2)\in\mathcal{R}_1}\frac{1}{N(a_0,a_1,a_2)},$$
where $\mathcal{R}_1$ is the union of good cubes.

Now we want to pass from $\mathcal{R}_1$ to $\mathcal{R}$, note that one of the variables in $\mathcal{R}\setminus\mathcal{R}_1$ has to be in a finite union intervals of length $\le M$, denote this union as $I_1(a_1,a_2)$. Using dyadic covering and Lemma \ref{norma}
$$\sum\limits_{(a_0,a_1,a_2)\in\mathcal{R}\setminus\mathcal{R}_1}\frac{1}{N(a_0,a_1,a_2)}\ll \sum_{2^k\ge (MN)^{1/2}}\frac{1}{2^{3k}}\sum\limits_{\max(a_1,a_2)\le 2^k}\sum\limits_{a_0\in I_1(a_1,a_2)}1\ll (M/N)^{1/2}= o(1).$$
Now we pass to the integral over $\alpha\in\mathcal{D},\: X^{1+\delta}<N(\alpha)\le X^{1+2\delta}$. For that we bound the contribution of $q\ll M^3$ and $B_{13}\ll M^2$ by the same argument as above. We have
$$S_0\ge \left(\frac{1}{3}+o(1)\right)C_1L(\delta)\prod_{p\le X^\delta}\left(1-\frac{g(p)}{p}\right)\left(1-\frac{1}{p}\right)^{-1}$$$$\times(\log x)^{-1}\sum\limits_{\substack{a_0,a_1,a_2\\\alpha\in\mathcal{D}\\ X^{1+\delta}<N(\alpha)\le X^{1+2\delta}}}\frac{1}{N(a_0,a_1,a_2)}.$$

By the prime ideal theorem the product is convergent because the contribution of degree 2 and 3 primes is negligible.

As for the sum, we know that $\mathcal{D}$ is a finite disjoint union of fundamental domains of the units action, hence for each principal ideal $(\alpha')$ there is a fixed number $d$ of ideals $(\alpha)$ in $\mathcal{D}$, that is Lemma \ref{D}. Hence

$$\sum\limits_{\substack{a_0,a_1,a_2\\\alpha\in\mathcal{D}\\ X^{1+\delta}<N(\alpha)\le X^{1+2\delta}}}\frac{1}{N(a_0,a_1,a_2)}=d\sum_{\substack{(\alpha')\\X^{1+\delta}<N(\alpha')\le X^{1+2\delta}}}\frac{1}{N(\alpha')},$$
where $\alpha'$ runs over all suitable principal ideals. And the resulting sum is
$$\sum_{\substack{(\alpha')\\X^{1+\delta}<N(\alpha')\le X^{1+2\delta}}}\frac{1}{N(\alpha')}=\frac{C_K}{\#\text{Cl}(K)}\delta\log x(1+o(1)),$$
where $C_K$ is the residue of $\zeta_K(s)$ at $s=1$ by lemma \ref{resid}.

Putting everything together we have
$$S_0\ge \frac{1}{3}\log(4/3)\frac{\delta d}{\#\text{Cl}(K)}C_0C_1C_KL(\delta)\prod_{p}\left(1-\frac{g(p)}{p}\right)\left(1-\frac{1}{p}\right)^{-1}(1+o(1))\gg 1.$$

\section{Acknowledgements}

Most of this paper was written as a mémoire during the M2 program Arithmétique, Analyse, Géométrie at Université Paris-Saclay. The author is grateful to his advisor, Cécile Dartyge, for her suggestions, patience, and support. He also thanks Kevin Destagnol and Étienne Fouvry for their encouragement and for suggesting working with Cécile.

\bibliographystyle{plain}
\bibliography{references}
\end{document}